\documentclass{amsart}
\usepackage{amssymb}
\usepackage{tikz}
\newcommand*{\DashedArrow}[1][]{\mathbin{\tikz [baseline=-0.25ex,-latex, dashed,#1] \draw [#1] (0pt,
0.5ex) -- (1.3em,0.5ex);}}

\usepackage{amscd}
\usepackage{empheq,amsmath}
\usepackage{graphicx}
\usepackage[cmtip,all]{xy}
\newtheorem{theorem}{Theorem}[section]
\newtheorem{lemma}[theorem]{Lemma}

\theoremstyle{definition}

\newtheorem{definition}[theorem]{Definition}

\theoremstyle{proposition}
\newtheorem{proposition}[theorem]{Proposition}

\newtheorem{remark}[theorem]{Remark}
\newtheorem{cor}[theorem]{Corollary}

\numberwithin{equation}{section}

\theoremstyle{main}
\newtheorem{main}{Theorem}

\begin{document}\title[Branched pull-back components]{Stability of branched pull-back projective foliations}

\author{Costa e Silva, W.}
\curraddr{IRMAR Universit\'{e} de Rennes 1, Campus de Beaulieu
35042 Rennes Cedex
France}
\email{wancossil@gmail.com}
\thanks{The first author is supported by Capes-Brazil.}

%




\begin{abstract}  We prove that, if $n\geq 3$, a singular foliation $\mathcal{F}$ on $\mathbb P^n$ which can be written as pull-back, where $\mathcal{G}$ is a  foliation in $ {\mathbb P^2}$ of degree $d\geq2$ with one or three invariant lines in general position and $f:{\mathbb P^n}\DashedArrow[->,densely dashed]{\mathbb P^2}$, $deg\left(f\right)=\nu\geq2,$ is an appropriated rational map, is stable under holomorphic deformations. As a consequence we conclude that the closure of the sets $\{\mathcal {F}= f^{*}\left(\mathcal{G}\right)\}$ are new irreducible components of the space of holomorphic foliations of certain degrees.\end{abstract}

\maketitle
\tableofcontents

\section{Introduction}
Let $\mathcal F$ be a holomorphic singular foliation on $\mathbb P^{n}$ of codimension $1$, \break$\Pi_{n}:\mathbb C^{n+1}\backslash \left\{0\right\} \to \mathbb P^{n}$  be the natural projection and $\mathcal F^*=\Pi_{n}^*\left(\mathcal F\right).$ It is known that $\mathcal F^*$ can be defined by an integrable $1-$form $\Omega=\sum_{j=0}^{n}A_{j}dz_{j}$ where the $A_{j}'s$ are homogeneous polynomials of the same degree $k+1$ satisfying the Euler condition:
\begin{equation}\label{rel-1}
\sum_{j=0}^{n}z_{j}A_{j}\equiv0.
\end{equation}
The singular set $S(\mathcal F)$ is given by $S(\mathcal F)=\left\{A_{0}=...=A_{n}=0\right\}$
and is such that codim$\left(S\left(\mathcal F\right)\right)\geq2$. The integrability condition is given by 
\begin{equation}\label{rel-2}
\Omega \wedge d\Omega=0.
\end{equation}

The form $\Omega$ will be called a homogeneous expression of $\mathcal F.$
The degree of $\mathcal F$ is, by definition, the number of tangencies (counted with multiplicities) of a generic linearly embedded $\mathbb P^1$ with $\mathcal F.$ If we denote it by $deg(\mathcal F)$ then $deg\left(\mathcal F\right)=k.$ The set of homogeneous $1$-forms which satisfy (\ref{rel-1}) and (\ref{rel-2}) will be denoted by $\tilde\Omega^1(n,k+1)$. We denote the space of foliations of a fixed degree $k$ in  $\mathbb P^n$ by $\mathbb{F}{\rm{ol}}\left(k,n\right)$. Due to the integrability condition and the fact that $S\left(\mathcal F\right)$ has codimension $\geq2$, we see that $\mathbb{F}{\rm{ol}}\left(k,n\right)$ can be identified with a Zariski's open set in the variety obtained by projectivizing the space of forms $\Omega$ which satisfy (\ref{rel-1}) and (\ref{rel-2}), i.e $\mathbb P\tilde\Omega^1(n,k+1)$. It is in fact an intersection of quadrics. To obtain a satisfactory description of $\mathbb{F}{\rm{ol}}\left(k;n\right)$ (for example, to talk about deformations) it would be reasonable to know the decomposition of $\mathbb{F}{\rm{ol}}\left(k;n\right)$ in irreducible components. This leads us to the following:
\vskip0.2cm
\noindent \textbf{{\underline{Problem}: \emph{Describe and classify the irreducible components of\break $\mathbb{F}{\rm{ol}}\left(k;n\right)$ $k\geq 3$ on ${\mathbb P^n}$, $n\geq 3.$}}}
\vskip0.2cm
One can exhibit some kind of list of components in every degree, but this list is incomplete.  In the  paper \cite{cln}, the authors proved that the space of holomorphic codimension one foliations of degree $2$ on  ${\mathbb P^n}, n\geq 3$, has six irreducible components, which can be described by geometric and dynamic properties of a generic element. We refer the curious reader to \cite{cln} and \cite{ln}  for a detailed description of them.
 There are known families of irreducible components in which the typical element is a pull-back of a foliation on ${\mathbb P^2}$ by a rational map. Given a generic rational map $f: {\mathbb P^n}  \DashedArrow[->,densely dashed    ]   {\mathbb P^2}$ of degree $\nu\geq1$, it can be written in homogeneous coordinates as $f=(F_0,F_1,F_2)$ where $F_0,F_1$ and $F_2$ are homogeneous polynomials of degree $\nu$. Now consider a foliation $\mathcal G$ on ${\mathbb P^2}$ of degree $d\geq2.$ We can associate to the pair $(f,\mathcal G)$ the pull-back foliation $\mathcal F=f^{\ast}\mathcal G$.
The degree of the foliation $\mathcal F$ is $\nu(d+2)-2$ as proved in \cite{clne}.
Denote by $PB(d,\nu;n)$ the closure in $\mathbb{F}{\rm{ol}}\left(\nu(d+2)-2,n\right)$, $n \geq 3$ of the set of foliations $\mathcal F$ of the form $f^{\ast}\mathcal G$. Since $(f,\mathcal G)\to f^{\ast}\mathcal G$ is an algebraic parametrization of $PB(d,\nu;n)$ it follows that $PB(d,\nu;n)$ is an unirational irreducible algebraic subset of $\mathbb{F}{\rm{ol}}\left(\nu(d+2)-2,n\right)$, $n \geq 3$. We have the following result:
\begin{theorem}
  $PB(d,\nu;n)$ is a unirational irreducible component of\break $\mathbb{F}{\rm{ol}}\left(\nu(d+2)-2,n\right);$ $n \geq 3$, $\nu\geq1$  and $d \geq 2$.
\end{theorem}
The case $\nu=1$, of linear pull-backs, was proven in \cite{caln}, whereas the case $\nu>1$, of nonlinear pull-backs, was proved in \cite{clne}. The search for new components of pull-back type was started in the Ph.D thesis of the author \cite{cs}. There we began to consider branched rational maps and foliations with algebraic invariant sets of positive dimensions.  

Let $\mathcal{F}$ be a holomorphic foliation on $\mathbb P^n$ which can be written as $\mathcal {F}= f^{*}\left(\mathcal{G}\right)$, where $\mathcal{G}$ is a  foliation in $ {\mathbb P^2}$ of degree $d\geq2$ with three invariant lines in general position, say $(XYZ)=0$, and $f:{\mathbb P^n}\DashedArrow[->,densely dashed]{\mathbb P^2}$, $deg\left(f\right)=\nu\geq2,$  $f=\left(F^\alpha_{0}:F^\beta_{1}:F^\gamma_{2}\right)$.
Denote by $PB(k,\nu,\alpha,\beta,\gamma)$ the closure in  $\mathbb{F}{\rm{ol}}\left(k,n\right)$, $n \geq 3$ of the set of foliations $\mathcal F$ of the form $f^{\ast}\mathcal G$. The degree of the foliation $\mathcal F$ is  ${k=\nu\left[\left(d-1\right)+\frac{1}{\alpha}+\frac{1}{\beta}+\frac{1}{\gamma}\right] - 2}$, as proved in \cite{cs}. Since $(f,\mathcal G)\to f^{\ast}\mathcal G$ is an algebraic parametrization of \break$PB(k,\nu,\alpha,\beta,\gamma)$ it follows that $PB(k,\nu,\alpha,\beta,\gamma)$ is an unirational irreducible algebraic subset of $\mathbb{F}{\rm{ol}}\left(k,n\right)$, $n \geq 3$. In \cite{cs} we proved the following result:
\begin{theorem}
 $PB(k,\nu,\alpha,\beta,\gamma)$ is a unirational irreducible component of \space$\mathbb{F}{\rm{ol}}\left(k,n\right)$ for all $n \geq 3$, $deg(F_{0}).\alpha = deg(F_{1}).\beta = deg(F_{2}).\gamma=\nu\geq2$, $(\alpha, \beta,\gamma) \in \mathbb N^{3}$ such that $1< \alpha < \beta < \gamma$ and $d \geq 2$. 
\end{theorem}

In this paper we continue looking for new components of branched pull back-type. In this direction  will extend the previous result to case where $\alpha=\beta\geq1$. We observe that in the case 
$\alpha=\beta>1$ we continue dealing with foliations in $\mathbb P^2$ with three invariant lines in general position. On the other hand, in the situation $\alpha=\beta=1$ we need to consider another set of foliations in $\mathbb P^2$. That is, we need foliations with one invariant line. Let us describe this last case: Let $\mathcal G$ be a foliation on $\mathbb P^2$ with one invariant straight line, say $\ell$. Consider coordinates $(X,Y,Z)\in \mathbb C^3$ such that  $\ell=\Pi_2(Z=0)$, where $\Pi_{2}:\mathbb C^{3}\backslash \left\{0\right\} \to \mathbb P^{2}$ is the natural projection. The foliation $\mathcal G$ can be represented in these coordinates by a polynomial $1$-form of the type  $ \Omega= ZA\left(X,Y,Z\right)dX+ZB\left(X,Y,Z\right)dY+C\left(X,Y,Z\right)dZ$ where by $(1)$ $XA+YB+C=0$. Let  $f: {\mathbb P^n}  \DashedArrow[->,densely dashed    ]   {\mathbb P^2}$ be a rational map represented in the coordinates\break $(X,Y,Z)\in \mathbb{C}^3$ and $W \in \mathbb C^{n+1}$ by $\tilde f=(F_{0},F_{1},F^{\gamma}_{2})$ where $F_{0},F_{1}$ and $F_{2} \in \mathbb C[W]$ are homogeneous polynomials without common factors satisfying $$deg(F_{0})=deg(F_{1})=\gamma.deg(F_{2})=\nu.$$
The pull back foliation $f^{*}(\mathcal G)$ is then defined by 
$$\tilde\eta_{[f,\mathcal G]}\left(W\right)=\left[F_{2}\left(A\circ F\right) dF_{0} + F_{2}\left(B\circ F\right) dF_{1} + {\gamma}\left(C\circ F\right) dF_{2}\right],$$
where each coefficient of  $\tilde\eta_{[f,\mathcal G]}\left(W\right)$ has degree $\Gamma=\nu\left[d+1+\frac{1}{\gamma}\right] - 1.$ The crucial point here is that the mapping $f$ sends the  hypersurface $(F_{2}=0)$ contained in its critical set over the line invariant by $\mathcal G$.

Let $PB\left(\Gamma-1,\nu,\alpha,\gamma\right)$ be the closure in $\mathbb{F}{\rm{ol}}\left(\Gamma-1,n\right)$ of the set $\left\{ \left[\tilde\eta_{[f,\mathcal G]} \right] \right\}$. It is an unirational irreducible algebraic subset of $\mathbb{F}{\rm{ol}}\left(\Gamma-1,n\right)$. We will return to this point in Section \ref{section4}.
We observe that the arguments for the cases $\alpha=\beta=1$ and $\alpha=\beta>1$ are similar. Hence we can unify the two situations in a unique statement. The main result of this work is:
\begin{main}\label{teob}$PB(\Gamma-1,\nu,\alpha,\gamma)$ is a unirational irreducible component of \break $\mathbb{F}{\rm{ol}}\left(\Gamma-1,n\right)$ for all $n \geq 3$, $ deg(F_{0}).\alpha = deg(F_{1}).\alpha =deg(F_{2}).\gamma=\nu\geq2$, such that $\alpha \geq1 $, $\gamma\geq2$, $\nu\geq2$ and $d \geq 2$ are integers. 
 \end{main}

\section {Branched rational maps}

Let $f : {\mathbb P^n}  \DashedArrow[->,densely dashed    ]   {\mathbb P^2}$ be a rational map and $\tilde{f}: {\mathbb C^{n+1}} \to {\mathbb C^3}$ is it natural lifting in homogeneous coordinates.
%
%
%
The \emph{indeterminacy locus} of $f$ is, by definition, the set $I\left(f\right)=\Pi_{n}\left(\tilde{f}^{-1}\left(0\right)\right)$. We characterize the set of rational maps used throughout this text as follows:

\begin{definition} We denote by $BRM\left(n,\nu,\alpha,\gamma\right)$ the set of maps \break $\left\{f: \mathbb P^n  \DashedArrow[->,densely dashed    ]   \mathbb P^2\right\}$ of degree $\nu$ given by $f=\left(F_{0}^\alpha:F_{1}^\alpha:F^\gamma_{2}\right)$  where $F_{0},F_{1}$ and $F_{2}$ are homogeneous polynomials without common factors, with $deg\left(F_{0}\right).\alpha=$ $deg\left(F_{1}\right).\alpha$ $=deg\left(F_{2}\right).\gamma$ $=\nu$, where $\nu\geq 2$, $\alpha\geq1$  and $\gamma \geq 2$ are integers.\end{definition}

Let us fix some coordinates $\left(z_{0},...,z_{n} \right)$ on $\mathbb C^{n+1}$ and $\left(X,Y,Z\right)$ on $\mathbb C^{3}$ and denote by $\left(F^\alpha_{0},F^\alpha_{1},F^\gamma_{2}\right)$ the components of $f$ relative to these coordinates. Let us note that the indeterminacy locus $I(f)$ is the intersection of the three hypersurfaces $(F_{0}=0)$, $(F_{1}=0)$ and $(F_{2}=0)$.

\begin{definition}\label{generic} We say that $f  \in BRM\left(n,\nu,\alpha,\gamma\right)$ is $generic$ if for all $p \in$ $\tilde{f}^{-1}\left(0\right)\backslash\left\{0\right\}$ we have $dF_{0}\left(p\right)\wedge dF_{1}\left(p\right)\wedge dF_{2}\left(p\right) \neq 0.$  
\end{definition}

This is equivalent to saying that $f  \in BRM\left(n,\nu,\alpha,\gamma\right)$ is $generic$ if $I(f)$ is the transverse intersection of the $3$ hypersurfaces $(F_{0}=0)$, $(F_{1}=0)$ and $(F_{2}=0)$.  As a consequence we have that  the set $I(f)$ is smooth. For instance, if $n=3$, $f$ is generic and $deg(f)=\nu$, then by Bezout's theorem $I\left(f\right)$ consists of $\frac{\nu^{3}}{\alpha^2\gamma}$ distinct points with multiplicity ${\alpha^2\gamma}$. If $n=4$, then $I\left(f\right)$ is a smooth connected algebraic curve in $\mathbb P^4$ of degree $\frac{\nu^{3}}{\alpha^2\gamma}$. In general, for $n \geq 4$, $I\left(f\right)$ is a smooth connected algebraic submanifold of $\mathbb P^n$ of degree $\frac{\nu^{3}}{\alpha^2\gamma}$ and codimension three. 

Denote $\nabla F_{k}= (\frac{\partial F_{k}}{\partial z_{0}},...,\frac{\partial F_{k}}{\partial z_{n}})$. Consider the derivative matrix 
{\tiny{ $$M=\begin{bmatrix}
 \alpha\left(F_{0}^{\alpha-1}\right)\nabla F_{0} \\ \alpha\left(F_{1}^{\alpha-1}\right)\nabla F_{1} \\  \gamma\left(F_{2}^{\gamma-1}\right)\nabla F_{2}  
 \end{bmatrix}.$$ }}
The critical set of $\tilde{f}$ is given by the points of $ \mathbb C^{n+1}\backslash \ {0}$ where rank$(M)\leq3$; it is the union of two sets. The first is given by   the set of $\left\{ P \in \mathbb C^{n+1}\backslash \ {0} \right\}=X_{1}$ such that the rank of the following matrix 
 {\tiny{$$N=\begin{bmatrix}
 \nabla F_{0} \\ \nabla F_{1} \\ \nabla F_{2}  
 \end{bmatrix}$$ }}is smaller than $3.$ The second is the subset
$$X_2=\left\{ P \in \mathbb C^{n+1}\backslash \left\{0\right\} |\left(F_{0}^{\alpha-1}\right)\left(F_{1}^{\alpha-1}\right)\left(F_{2}^{\gamma-1}\right)\left(P\right)=0 \right\}.$$

Denote $P\left(f\right)=\Pi_{n}\left(X_{1} \cup X_{2}\right)$. The set of generic maps will be denoted by $Gen\left(n,\nu,\alpha,\gamma\right)$. We state the following result whose proof is standard in algebraic geometry:
\begin{proposition} $Gen\left(n,\nu,\alpha,\gamma\right)$ is a Zariski dense subset of $BRM\left(n,\nu,\alpha,\gamma\right)$.
\end{proposition}

Once the case of foliations which are pull-backs of three invariant straight have been already discussed in \cite{cs}. We will concentrate only on the case where $\alpha=1$. The case $\alpha>1$ is obtained following the same ideas. 
\section{Foliations with one invariant line}

\subsection{Basic facts} 
Denote by $I_1(d,2)$ the set of the holomorphic foliations on $\mathbb P^{2}$ of degree $d \geq 2$ that leaves the line $Z=0$ invariant. We observe that any foliation which has $1$ invariant straight line can be carried to one of  these by a linear automorphism of $\mathbb P^{2}$.
The relation
$XA+YB+C=0$
enables to parametrize $I(d,2)$  as follows
\begin{align*}
  {\rm {H}}^0(\mathbb P^2,\mathcal O_{\mathbb P^2}(d-1))^{\times 2}  &\to {\rm {H}}^0(\mathbb P^2,\mathcal O_{\mathbb P^2}(d-1))^{\times 3} \\
                         (A,B) &\mapsto (A,B,-XA-YB).
\end{align*}
We let the group of linear automorphisms of $\mathbb P^{2}$ act on $I_1(d,2)$. After this procedure we obtain a set of foliations of degree $d$ that we denote by $Il_{1}(d,2)$.\par We are interested in making deformations of foliations and for our purposes we need a subset of  $Il_{1}(d,2)$ with good properties (foliations having few algebraic invariant curves and only hyperbolic singularities). We explain this properties in detail. Let $q \in U$ be an isolated singularity of a foliation $\mathcal G$ 
defined on an open subset of $U\subset \mathbb C^{2}.$ We say that $q$ is $nondegenerate$ if there exists a holomorphic vector field $X$ tangent to  $\mathcal G$ in a neighborhood of $q$ such that $DX(q)$ is nonsingular. In particular $q$ is an isolated singularity of $X.$
Let $q$ be a nondegenerate singularity of $\mathcal G$. The \textit{characteristic numbers} of $q$ are the quotients $\lambda$ and $\lambda^{-1}$ of the eingenvalues of $DX(q)$, which do not depend on the vector field $X$ chosen. If $\lambda \notin \mathbb Q_{+}$ then $\mathcal G$ exhibits exactly two (smooth and transverse) \textit{local separatrices} at $q$,  $S_{q}^{+}$ and $S_{q}^{-}$ with eigenvalues $\lambda_{q}^{+}$ and $\lambda_{q}^{-}$ and which are tangent to the characteristic directions of a vector field $X$. The characteristic numbers (also called Camacho-Sad index) of these local separatrices are given by 
$$I(\mathcal G,S_{q}^{+})=\frac{\lambda_{q}^{-}}{\lambda_{q}^{+}} \ \text{and} \
I(\mathcal G,S_{q}^{-})=\frac{\lambda_{q}^{+}}{\lambda_{q}^{-}}.$$
The singularity is \textit{hyperbolic} if the characteristic numbers are nonreal. We introduce the following spaces of foliations:

\begin{enumerate}
\item [(1)]$ND(d,2)=\{ \mathcal{G} \in \mathbb{F}{\rm{ol}}(d,2);$ the singularities of $\mathcal G$ are nondegenerate$\},$
\item[(2)] $\mathcal{H}(d,2)=\{\mathcal{G} \in ND(d,2);$ any characteristic number $\lambda$ of $\mathcal{G}$ satisfies $\lambda \in \mathbb C \backslash \mathbb R\}.$
\end{enumerate}
It is a well-known fact  \cite{ln1} that $\mathcal{H}(d,2)$ contains an open and dense subset of $\mathbb{F}{\rm{ol}}(d,2)$. Denote by $A(d)=Il_{1}(d,2)\cap \mathcal{H}(d,2).$ Observe that $A(d)$ is a Zariski dense subset of $Il_{1}(d,2)$. Concerning the set $ND(d,2)$, we have the following result, proved in \cite{ln1}.

\begin{proposition}
Let  $\mathcal G_{0} \in ND(d,2)$. Then $\#Sing(\mathcal G_{0})=d^2+d+1=N(d)$. Moreover if $Sing(\mathcal G_{0})=\{p_{1}^0,...,p_{N}^0\}$ where $p_{i}^0\neq p_{j}^0$ if $i\neq j$, then there are connected neighborhoods $U_{j} \ni p_{j}$, pairwise disjoint, and holomorphic maps $\phi_{j}:\mathcal U\subset ND(d,2) \to U_{j}$, where $\mathcal U \ni \mathcal G_{0}$ is an open neighborhood, such that for $\mathcal G \in \mathcal U$, $(Sing(\mathcal G)\cap U_{j})=\phi_{j}(\mathcal G)$  is a nondegenerate singularity.  In particular, $ND(d,2)$ is open in $\mathbb{F}{\rm{ol}}(2,d)$. Moreover, if $\mathcal G_{0} \in \mathcal{H}(d,2)$ then the two local separatrices as well as their associated eigenvalues depend analytically on $\mathcal {G}$.
\end{proposition}
\vskip0.2cm
 In the paper \cite{lnssc} which is related to the topological rigidity of foliations on $\mathbb P^2$ in the spirit of Ilyashenko's works. The authors have proved the following useful result see\cite[Theorem 3, p.385]{lnssc}. \begin{theorem} Let $d\geq 2$. There exists an non empty open and dense subset $M\left(d\right)\subset A\left(d\right)$, such that if $\mathcal G \in M\left(d\right)$  then the only algebraic invariant curve of $\mathcal G$ is the line. 
\end{theorem}

\section{Ramified pull-back components - Generic conditions}\label{section4}

Let us fix a coordinate system $(X,Y,Z)$ on $\mathbb P^2$ and denote by $\ell$ the straight line that corresponds to the plane $Z=0$ in $\mathbb C^3$, respectively.
Let us denote by $\tilde M\left(d\right)$ the subset $M\left(d\right)\cap I_1(d,2)$.
\begin{definition} Let $f \in Gen\left(n,\nu,1,\gamma\right)$. We say that  
$\mathcal G \in M\left(d\right)$ is in generic position with respect to $f$ if $\left[Sing\left(\mathcal G\right)\cap Y_{2}\right]=\emptyset,$ where 
$$Y_{2}(f)=Y_{2}:=\displaystyle\Pi_2\left[\tilde{f}\left\{w \in \mathbb C^{n+1} |dF_{0}\left(w\right)\wedge dF_{1}\left(w\right)\wedge dF_{2}\left(w\right) = 0\right\}\right]$$ and  $\ell$ is $\mathcal G$-invariant.
\end{definition}

In this case we say that $\left(f,\mathcal{G}\right)$ is a generic pair. In particular, when we fix a map $f\in Gen(n,\nu,1,\gamma)$ the set $\mathcal{A}=\left\{\mathcal{G} \in M\left(d\right) | Sing\left(\mathcal G\right) \cap Y_{2}(f)= \emptyset \right\}$ is an open and dense subset in $M(d)$ \cite{lnsc}, since $VC(f)$ {is an algebraic curve in}  ${\mathbb P^2}.$ The set $U_{1}:=\{ (f,\mathcal G) \in Gen(n,\nu,1,\gamma)\times\tilde M\left(d\right)| Sing\left(\mathcal G\right)\cap Y_{2}(f)= \emptyset \}$ is an open and dense subset of $Gen(n,\nu,1,\gamma)\times\tilde M\left(d\right)$. Hence the set $\mathcal W:=\left\{\tilde\eta_{[f,\mathcal G]}| \left(f,\mathcal G\right)\in U_{1}\right\}$ is an open and dense subset of $PB\left(\Gamma-1,\nu,1,\gamma\right)$.
    \begin{proposition} \label{graupargenerico}If $\mathcal F$ comes from a generic pair, then the degree of $\mathcal F$ is $$\nu\left[d+1+\frac{1}{\gamma}\right] - 2.$$ 
   \end{proposition}
  The proof of this fact can be obtained as in the case treated in \cite{cs}.

Consider the set of foliations $Il_{1}\left(d,2\right)$, $d\geq 2,$ and the following map:
\begin{eqnarray*}
\Phi:BRM\left(n,\nu,1,\gamma\right) \times Il_{1}\left(d,2\right) &\to&\mathbb{F}{\rm{ol}}\left(\Gamma-1,n\right)\\
\left(f,\mathcal G\right) &\to& f ^{\ast}\left(\mathcal G\right)=\Phi\left(f,\mathcal G\right).
\end{eqnarray*}  
The image of $\Phi$ can be written as:  $$\Phi\left(f,\mathcal G\right)=\left[F_{2}\left(A\circ F\right) dF_{0} + F_{2}\left(B\circ F\right) dF_{1} + {\gamma}\left(C\circ F\right) dF_{2}\right].$$ Recall that $\Phi\left(f,\mathcal G\right)=\tilde\eta_{[f,\mathcal G]}$. More precisely, let $PB(\Gamma-1,n,\nu,1,1,\gamma)$ be the closure in $\mathbb{F}{\rm{ol}}\left(\Gamma-1,n\right)$ of the set of foliations $\mathcal{F}$ of the form $f^*\left(\mathcal{G}\right)$, where $f\in BRM\left(n,\nu,1,\gamma\right)$ and $\mathcal G \in Il_{1}(2,d).$ Since $BRM\left(n,\nu,1,\gamma\right)$ and $Il_{1}(2,d)$ are irreducible algebraic sets and the map $\left(f,\mathcal{G}\right) \to f^*\left(\mathcal{G}\right) \in \mathbb{F}{\rm{ol}}\left(\Gamma-1,n\right)$ is an algebraic parametrization of $PB(\Gamma-1,\nu,1,\gamma)$, we have that $PB(\Gamma-1,\nu,1,\gamma)$ is an irreducible algebraic subset of $\mathbb{F}{\rm{ol}}\left(\Gamma-1,n\right)$. Moreover, the set of generic pull-back foliations $\left\{\mathcal{F}; \mathcal{F} = f^*(\mathcal{G}),\ \text{where} \ \left(f,\mathcal{G}\right)\right.$ is a generic pair$\left.\right\}$ is an open (not Zariski) and dense subset of $PB(\Gamma-1,\nu,1,\gamma)$ for $\gamma\geq2  \in \mathbb N$, $\nu\geq2  \in \mathbb N$ and $d \geq 2  \in \mathbb N$. 

\section{Description of generic ramified pull-back foliations on $\mathbb P^n$}
\subsection{The Kupka set}\label{section5.1}

Let $\tau$ be a singularity of $\mathcal{G}$ and $V_{\tau}=\overline{f^{-1}(\tau)}$.  If $(f,\mathcal G)$ is a generic pair then $V_{\tau}\backslash I(f)$ is contained in the Kupka set  of $\mathcal F$. As an example we detail the case where $\tau$ is a singularity over the invariant line, say $\tau=[1:0:0]$. Fix $p\in V_{\tau}\backslash I(f)$. There exist local analytic coordinate systems such that $f(x,y,z)=(x,y^{\gamma})=(u,v)$. Suppose that $\mathcal {G}$ is represented by the 1-form $\omega$; the hypothesis of $\mathcal G$ being of Hyperbolic-type implies that we can suppose $\omega(u,v) =  \lambda_{1}u(1+R(u,v))dv - \lambda_{2}vdu,$ where $\frac{\lambda_{2}}{\lambda_{1}} \in \mathbb{C}\backslash \mathbb{R}$. We obtain $\tilde \omega(x,y)=f^{\ast}(\omega)=(y^{\gamma -1})(\lambda_{1}\gamma x(1+R(x,y^{\gamma})dy-\lambda_{2}ydx)=(y^{\gamma -1})\hat \omega(x,y)$ and so $d\hat \omega(p)\neq0$. Therefore if $p$ is as before it belongs to the Kupka-set of $\mathcal {F}$. For the other points the argumentation is analogous. This is the well known Kupka-Reeb phenomenon, and we say that $p$ is contained in the Kupka-set of $\mathcal {F}$. It is known that this local product structure is stable under small perturbations of $\mathcal F$ for instance, see \cite{kupka},\cite{gln}. 
 
\subsection{Generalized Kupka and quasi-homogeneous singularities}\label{section5.2}
 In this section we will recall the quasi-homogeneous singularities of an  integrable holomorphic $1$-form.  They appear in the indeterminacy set of $f$ and play a central role in great part of the proof of Theorem  B. 
\begin{definition} Let  $\omega$ be an holomorphic  integrable 1-form defined in a neighborhood of $p \in {\mathbb{C}^3}$. We say that $p$ is a Generalized Kupka(GK) singularity of $\omega$ if $\omega(p)=0$ and either 
$d\omega(p)\neq0$ or $p$ is an isolated zero of $d\omega$.   
\end{definition}
 \par
Let $\omega$ be an integrable $1$-form in a neighborhood of $p  \in {\mathbb{C}^3}$  and $\mu$ be a holomorphic $3$-form such that $\mu({p}) \neq 0$. Then $d\omega=i_{\mathcal Z}(\mu)$ where $\mathcal Z$ is a holomorphic vector field. 
\begin{definition} We say that $p$ is a quasi-homogeneous singularity of $\omega$ if  $p$ is an isolated singularity of $\mathcal Z$ and the germ of $\mathcal Z$ at $p$ is nilpotent, that is, if $L=D\mathcal Z(p)$ then all eigenvalues of $L$ are equals to zero.
\end{definition}
 This definition is justified by the following result that can be found in \cite{ln1} or \cite{ccgl}:
\begin{theorem} \label{teo5.3}
Let $p$ be a quasi-homogeneous singularity of an holomorphic integrable 1-form $\omega.$  Then there exists two holomorphic vector fields $S$ and $\mathcal Z$ and a local chart  $U:=(x_{0},x_{1},x_{2})$  around $p$ such that  $x_{0}(p)=x_{1}(p)=x_{2}(p)=0$ and:
\begin{enumerate}
\item[(a)] $\omega$ = $\lambda i_{S}i_{\mathcal Z}(dx_0 \wedge dx_{1} \wedge dx_{2}),$ $\lambda \in \mathbb Q_{+}$ 
$d\omega= i_{\mathcal Z}(dx_{0} \wedge dx_{1} \wedge dx_{2})$ and $\mathcal Z=(rot (\omega))$;
\item[(b)] $S=p_{0}x_{0} \frac{\partial}{\partial x_{0}}+ p_{1}x_{1} \frac{\partial}{\partial x_{1}}+p_{2}x_{2} \frac{\partial}{\partial x_{2}} $, where, $p_{0}$, $p_{1}$, $p_{2}$ are positive integers with $g.c.d (p_{0},p_{1},p_{2})=1$;
\item[(c)]  $p$ is an isolated singularity for $\mathcal Z$, $\mathcal Z$ is polynomial in the chart\break $U:=(x_{0},x_{1},x_{2})$ and  $[S,\mathcal Z]=\ell \mathcal Z$, where $\ell \geq1$.
\end{enumerate}
\end{theorem}
 \begin{definition}
 Let $p$ be a quasi-homogeneous singularity of $\omega.$ We say that it is of the type $(p_{0}:p_{1}:p_{2};\ell)$, if  for some local chart and vector fields $S$ and $\mathcal Z$  the properties $(a),(b)$ and $(c)$ of the {\textrm{Theorem}} \ref{teo5.3} are satisfied.
 \end{definition}

 We can now state the stability result, whose proof can be found in\break \cite{ccgl}:
 \begin{proposition}\label{prop5.5} Let $(\omega_{s})_{s \in \Sigma}$ be a holomorphic family of integrable $1$-forms defined in a neighborhood of a compact ball  $B= \{ {z \in {\mathbb{C}^3} ; |z|} \leq \rho \}$, where $\Sigma$ is a neighborhood of $0 \in {\mathbb{C}^k}.$  Suppose that all singularities of $\omega_{0}$ in $B$ are $GK$ and that $sing(d\omega_{0})\subset int(B)$. Then there exists $\epsilon >0$ such that if $s \in B(0,\epsilon)\subset\Sigma,$ then all singularities of $\omega_{s}$ in $B$ are GK. Moreover, if $0 \in B$ is a quasi-homogeneous singularity of type $(p_{0}:p_{1}:p_{2};\ell)$ then there exists a holomorphic map $B(0,\epsilon) \ni s \mapsto z(s)$, such that $z(0)=0$ and $z(s)$ is a $GK$ singularity of  $\omega_{s}$ of the same type (quasi-homogeneous of the type $(p_{0}:p_{1}:p_{2};\ell)$, according to the case).
 \end{proposition}
 
 Let us describe  $\mathcal{F} = f^*(\mathcal{G})$ in a neighborhood of a point $p \in I(f).$
\noindent It is easy to show that there exists a local chart $(U,(x_0,x_1,x_2,y)\in \mathbb C^3\times \mathbb C^{n-2})$  around $p$ such that the lifting $\tilde f$ of $f$ is of the form $\tilde f|_{U}=(x_{0},x_{1},x^{\gamma} _{2}):U \to {\mathbb{C}^3}$. In particular $\mathcal {F}|_{U(p)}$ is represented by the $1$-form 
\begin{equation}\label{eta}
\eta (x_{0},x_{1},x_{2},y) = x_{2}.A(x_{0},x_{1},x^{\gamma} _{2})dx_{0} 
+ x_{2}.B(x_{0},x_{1},x^{\gamma} _{2})dx_{1}\end{equation}\begin{equation*} + \gamma C(x_{0},x_{1},x^{\gamma} _{2})dx_{2}. 
\end{equation*}

Let us now obtain the vector field $S$ as in Theorem \ref{teo5.3}. Consider the radial vector field $R=X \frac{\partial}{\partial X}+ Y \frac{\partial}{\partial Y}+ Z \frac{\partial}{\partial Z}$. Note that in the coordinate system above it transforms into
$$ x_{0}  \frac{\partial}{\partial x_{0}}+x_{1}  \frac{\partial}{\partial x_{1}}+\frac{1}{\gamma}x_{2}  \frac{\partial}{\partial x_{2}}.$$ Since the eigenvalues of $S$ have to be integers, after a multiplication by ${\gamma}$ we obtain $$S={\gamma} x_{0} \frac{\partial}{\partial x_{0}}+ {\gamma}x_{1} \frac{\partial}{\partial x_{1}}+x_{2} \frac{\partial}{\partial x_{2}}.$$
 Let us concentrate in the case $n=3$.

\begin{lemma}\label{liecartan} If $\eta$ and $S$ are as above
then we have  $L_{S}\eta=[1+\gamma(1+d)]\eta$.
\end{lemma}

\begin{proof} We just have to use Cartan's formula for the Lie's derivative, $L_{S}\eta=i_Sd\eta+d(i_S\eta)$. The details are left for the reader.
\end{proof}
\begin{lemma} \label{lema5.6}If $p\in I(f)$ then $p$ is a quasi-homogeneous singularity of $\eta$.    \end{lemma} 
  \begin{proof}
First of all note that $i_{S}\eta=0$. From the computations  obtained in lemma \ref{liecartan}, we have that $L_{S}\eta=m\eta,$ where $m=[1+\gamma(1+d)]$. This implies that the singular set of $\eta$ is invariant under the flow of $S$.  
The vector field $\mathcal Z$ such that $\eta=i_{S}i_{\mathcal Z}(dx_0 \wedge dx_{1} \wedge dx_{2})$ is given by
$$\mathcal Z=\mathcal Z_{0}(x_{0},x_{1},x_{2}) \frac{\partial}{\partial x_{0}}+ \mathcal Z_{1}(x_{0},x_{1},x_{2}) \frac{\partial}{\partial x_{1}}+ \mathcal Z_{2}(x_{0},x_{1},x_{2}) \frac{\partial}{\partial x_{2}}$$
where for $i=0,1$ we have $\mathcal Z_{i}(x_{0},x_{1},x_{2})= \tilde{A}_{i}(x_{0},x_{1},x^{\gamma} _{2})$  and $\mathcal Z_{2}(x_{0},x_{1},x_{2})= x_{2}.\tilde{A}_{2}(x_{0},x_{1},x^{\gamma} _{2})$ moreover for $i=0,1$ the polynomials $\tilde A_{i}(0,0,0)=0$ and $\tilde A_{2}(0,0,0)=0$.  
We observe that these polynomials are not unique. On the other hand, they have to satisfy the following relations:
\begin{equation*}\label{polinômios 1}
A(x_{0},x_{1},x^{\gamma} _{2})=\gamma x_1\tilde{A}_{2}(x_{0},x_{1},x^{\gamma} _{2})-\tilde A_{1}(x_{0},x_{1},x^{\gamma} _{2})
\end{equation*}
\begin{equation*}\label{polinômios 2}
B(x_{0},x_{1},x^{\gamma} _{2})=\tilde{A}_{0}(x_{0},x_{1},x^{\gamma} _{2})-\gamma x_{0}\tilde A_{2}(x_{0},x_{1},x^{\gamma} _{2})
\end{equation*}
\begin{equation*}\label{polinômios 3}
C(x_{0},x_{1},x^{\gamma} _{2})= x_{0}\tilde{A}_{1}(x_{0},x_{1},x^{\gamma} _{2})-x_{1}\tilde A_{1}(x_{0},x_{1},x^{\gamma} _{2})
\end{equation*}

We must show that the origin is an isolated singularity of $\mathcal Z$ and all eigenvalues of $D\mathcal Z(0)$ are $0$. By straightforward computation we find that the Jacobian matrix $D\mathcal Z(0)$ is the null matrix, hence all its eigenvalues are null. Since all singular curves of $\mathcal F$ in a neighborhood $(U,(x_0,x_1,x_2))$ of $0$ are of Kupka type, as proved in Section \ref{section5.1}, it follows that the origin is an isolated sigularity of $\mathcal Z$. Note that the unique singularities of $\eta$ in the neighborhood $(U,(x_0,x_1,x_2))$ of $0$ come from $\tilde f^\ast Sing(\mathcal G)$; this follows from the fact that  $Sing\left(\mathcal G\right) \cap (VC(f) \backslash \ell)= \emptyset$.
  On the other hand we have seen that $(f)^{-1}(sing(\mathcal G))\backslash I(f)$ is contained in the Kupka set of $\mathcal{F}$. Hence the point $p$  is an isolated singularity of $d\eta$ and thus an isolated singularity of $\mathcal Z$.
 \end{proof} 
 
 As a consequence, in the case $n=3$ any $p\in I(f)$ is a quasi-homogeneous singularity of type $\left [{\gamma}:{\gamma}:1\right ]$. In the case $n\geq4$ the argument is analogous. Moreover, in this case there will be a local structure product near any point $p\in I(f)$. In fact in the case $n\geq4$ we have:
\begin{cor} Let $(f,\mathcal G)$ be a generic pair. Let $p \in I(f)$ and $\eta$ an 1-form defining $\mathcal F$ in a neighborhood of $p$. Then there exists a 3-plane $\Pi \subset  \mathbb C^{n}$ such that $d(\eta)|_{\Pi}$ has an isolated singularity at $0 \in{\Pi}. $
\end{cor}
\begin{proof} Immediate from the local product structure.
\end{proof}

\subsection{Deformations of the singular set}
In this section we give some auxiliary lemmas which assist in the proof of Theorem \ref{teob}. We have constructed an open and dense subset $\mathcal W$ inside {$PB(\Gamma-1,\nu,1,1,\gamma)$} containing the generic pull-back foliations.  We will show that for any foliation $\mathcal{F} \in \mathcal W$ and any germ of a holomorphic family of foliations $(\mathcal{F}_{t})_{t \in (\mathbb{C},0)}$ such that $\mathcal{F}_{0}=\mathcal{F}$ we have $\mathcal{F}_{t} \in PB(\Gamma-1,\nu,1,1,\gamma)$ for all $t \in (\mathbb{C},0).$
\begin{lemma} \label{Lemma5.8} There exists a germ of isotopy of class $C^{\infty}$, $(I(t))_{t \in (\mathbb {C},0)}$ having the following properties: 
   \begin{enumerate}
\item[(i)]  $I(0)=I(f_{0})$ and $I(t)$ is algebraic and smooth of codimension $3$ for all  ${t \in ({\mathbb {C}}, 0)}.$
 \item[(ii)] For all $p \in I(t)$, there exists a neighborhood $U(p,t)=U$ of $p$ such that $\mathcal{F}_{t}$ is equivalent to the product of a regular foliation of codimension $3$ and a singular foliation $\mathcal F_{p,t}$ of codimension one given by the 1-form $\eta_{p,t}$.
 \end{enumerate}
 \begin{remark}
The family of $1$-forms $\eta_{p,t}$, represents the quasi-homogeneous foliation given by the Proposition \ref{prop5.5}. 
  \end{remark}
\end{lemma}
 \begin{proof} See \cite[lema 2.3.2, p.81]{ln}.
 \end{proof}

\begin{remark}
In the case $n>3$, the variety $I(t)$ is connected since $I(f_{0})$ is connected. The local product structure in $I(t)$ implies that the transversal type of $\mathcal{F}_{t}$ is constant. In particular, $\mathcal F_{p,t}$, does not depend on $p \in I(t)$. In the case $n=3$, $I(t)=p_{1}(t),..., p_{j}(t),...,p_{\frac{\nu^{3}}{\gamma}}(t)$ and we can not guarantee a priori that $\mathcal F_{p_{i},t}=\mathcal F_{p_{j},t}$, if $i \neq j.$
\end{remark}
The singular set of  $\mathcal{G}_{0}$ can be divided in two subsets $\mathcal S_{W}(\mathcal G_{0})$,  $\mathcal S_{\ell}(\mathcal G_{0})$.  We know that  $\# \mathcal S_{W}(\mathcal G_{0})=d^{2}$, $\# \mathcal S_{\ell}(\mathcal G_{0})=(d+1)$.  Let $\tau \in Sing( \mathcal G_{0})$ and $K(\mathcal F_0)=\cup_{{\tau\in Sing( \mathcal G_{0})}}V_{\tau}\backslash I(f_{0})$ where $V_{\tau}=\overline{f_{0}^{-1}(\tau)}.$ 
As in Lemma \ref{Lemma5.8}, let us consider a representative of the germ $(\mathcal F_{t})_{t}$, defined on a disc $D_{\delta}:=(|t|<\delta).$ 
  
\begin{lemma} 
There exist $\epsilon >0$ and smooth isotopies $\phi_{\tau}:D_{\epsilon}\times V_{\tau} \to \mathbb P^n, \tau \in Sing(\mathcal G_{0})$, such that $V_{\tau}(t)=\phi_{\tau}(\{t\}\times V_{\tau})$ satisfies:
\begin{enumerate}
\item[(a)] $V_{\tau}(t)$ is an algebraic subvariety of codimension two of $\mathbb P^n$ and  $V_{\tau}(0) = V_{\tau}$ for all $ \tau \in Sing(\mathcal G_{0})$ and for all $t \in D_{\epsilon}.$
\item[(b)] $I (t) \subset V_{\tau}(t)$ for all $\tau \in Sing(\mathcal G_{0})$ and for all $t \in D_{\epsilon} $. Moreover, if $\tau \neq \tau'$, and $\tau, \tau' \in Sing(\mathcal G_{0})$,  we have $V_{\tau}(t)  \cap V_{\tau'}(t) = I(t)$ for all $t \in D_{\epsilon}$ and the intersection is transversal.
\item[(c)]   $V_{\tau}(t) \backslash I(t)$ is contained in the Kupka-set of $\mathcal{F}_t$ for all $ \tau \in Sing(\mathcal G_{0})$ and for all $t \in D_{\epsilon}.$ In particular, the transversal type of $\mathcal F_{t}$ is constant along $V_{\tau}(t) \backslash I(t)$. 
\end{enumerate}
\end{lemma}
\begin{proof}  See \cite[lema 2.3.3, p.83]{ln}.
 \end{proof}

\section{Proof of theorem A}
\subsection{End of the proof of Theorem \ref{teob}}\label{subsection5.4} 

We divide the end of the proof of Theorem \ref{teob} in two parts. In the first part we construct a family  of rational maps $f_{t}:{\mathbb{P}^n \DashedArrow[->,densely dashed    ]   \mathbb{P}^2}$, $f_{t} \in Gen(n,\nu,1,\gamma)$, such that $(f_{t})_{t \in D_{\epsilon}}$ is a deformation of $f_{0}$ and the subvarieties $V_\tau$, $\tau \in Sing(\mathcal{G}_0)$, are fibers of $f_t$ for all $t$. In the second part we show that there exists a family of foliations $(\mathcal {G}_{t})_{t\in D_{\epsilon}}, \mathcal {G}_{t} \in  \mathcal A$ (see Section \ref{section4}) such that  $\mathcal{F}_{t}= f_{t}^{ \ast}(\mathcal{G}_{t})$ for all $t\in D_{\epsilon}$. 

\subsubsection{Part 1}\label{part1} Let us define the family of candidates that will be a deformation of the mapping $f_{0}.$   Set $V_{a}=\overline{{f}_{0}^{-1}(a)}$, $V_{b}=\overline{{f}_{0}^{-1}(b)}$, $V_{c}=\overline{{f}_{0}^{-1}(c)}$, where $a=[0:0:1]$, $b=[0:1:0]$ and  $c=[1:0:0]$ and denote by $V_{\tau^\ast}=\overline{{f}_{0}^{-1}(\tau^\ast)}$, where $\tau^\ast \in Sing(\mathcal G_{0})\backslash\{a,b,c\}$. In this coordinate system the points $b$ and $c$ belong to $\ell$.
  \begin{proposition}\label{recupmapas}
Let $(\mathcal{F}_{t})_{t \in D_{\epsilon}}$ be a deformation of $\mathcal F_{0}= f_{0}^*(\mathcal G_{0})$, where $(f_{0}, \mathcal G_{0})$ is a generic pair, with $\mathcal G_{0} \in \mathcal A$, $f_0 \in Gen\left(n,\nu,1,\gamma\right)$ and $deg(f_{0})=\nu\geq 2$. Then there exists a deformation $({f}_{t})_{t \in D_{\epsilon}}$ of $f_{0}$ in $Gen\left(n,\nu,1,\gamma\right)$ such that:
\begin{enumerate}
\item[(i)]$V_{a}(t),V_{b}(t)$ and $V_{c}(t)$ are fibers of $({f}_{t})_{t \in D_{\epsilon'}}$.
\item[(ii)] $I(t)=I(f_{t}), \forall {t \in D_{\epsilon'}}$.
    \end{enumerate}
  \end{proposition}  
\begin{proof} Let $\tilde f_{0}=(F_{0},F_{1},F^\gamma_{2}): \mathbb C^{n+1} \to\mathbb C^{3}$ be the homogeneous expression of $f_{0}$. Then $V_{c}$, $V_{b}$, and $V_{a}$ appear as the complete intersections $(F_{1}=F_{2}=0)$, $(F_{0}=F_{2}=0)$, and $(F_{0}=F_{1}=0)$ respectively. Hence $I(f_{0})=V_{a} \cap V_{b} =V_{a} \cap V_{c}= V_{b} \cap V_{c}$. It follows from \cite[section 4.6, p.235-236]{Ser0} that $V_{a}(t)$ is a complete intersection, say $V_{a}(t)=(F_{0}(t)=F_{1}(t)=0)$, where $(F_{0}(t))_{t \in D_{\epsilon'}}$ and $(F_{1}(t))_{t \in D_{\epsilon'}}$ are deformations of $F_{0}$ and $F_{1}$ and $ D_{\epsilon'}$ is a possibly smaller neighborhood of $0$. Moreover, $F_{0}(t)=0$ and $F_{1}(t)=0$ meet transversely along $V_{a}(t)$. In the same way, it is possible to define $V_{c}(t)$ and $V_{b}(t)$ as complete intersections, say $(\hat F_{1}(t)=F_{2}(t)=0)$ and $(\hat F_{0}(t)=\hat F_{2}(t)=0)$ respectively, where $(F_{j}(t))_{t \in D_{\epsilon'}}$ and $(\hat F_{j}(t))_{t \in D_{\epsilon'}}$ are deformations of $F_{j}$, $0\leq j\leq2$. 

We will prove that we can find polynomials $P_{0}(t)$, $P_{1}(t)$ and $P_{2}(t)$ such that $V_{c}(t)=(P_{1}(t)=P_{2}(t)=0)$, $V_{b}(t)=(P_{0}(t)=P_{2}(t)=0)$ and $V_{a}(t)=(P_{0}(t)=P_{1}(t)=0)$. Observe first that since $F_{0}(t), F_{1}(t)$ and $F_{2}(t)$ are near $F_{0}$, $F_{1}$ and $F_{2}$ respectively, they meet as a regular complete intersection at:
$$J(t)=(F_{0}(t)=F_{1}(t)=F_{2}(t)=0) = V_a(t) \cap (F_2(t)=0).$$
Hence $J(t) \cap (\hat F_{1}(t)=0)=V_{c}(t)\cap V_{a}(t)=I(t)$, which implies that $I(t)\subset J(t).$  Since $I(t)$ and $J(t)$ have $\frac{\nu^{3}}{\gamma}$ points, we have that  $I(t)=J(t)$ for all $t \in D_{\epsilon'}$. 
\begin{remark}In the case $n\geq 4$, both sets are codimension-three smooth and connected submanifolds of $\mathbb P^n$, implying again that $I(t)=J(t).$ In particular, we obtain that $$I(t)=(F_{0}(t)=F_{1}(t)=F_{2}(t)=0)\subset (\hat F_{j}(t)=0), 0\leq j\leq2.$$
 \end{remark}  
 
We will use the following version of Noether's Normalization Theorem (see \cite{ln} p 86):

\begin{lemma} (Noether's Theorem) Let $G_{0},...,G_{k}  \in \mathbb {C}[z_{1},...,z_{m}]$ be homogeneous polynomials where $0 \leq k \leq m$ and $m \geq 2$, and $X=(G_{0}=...=G_{k}=0)$. Suppose that the set $Y:=\{p \in X | dG_{0}(p) \wedge...\wedge dG_{k}(p)=0\}$ is either $0$ or $\emptyset$. If $G  \in \mathbb {C}[z_{1},...,z_{m}]$ satisfies $G|_{X} \equiv 0$, then  $G$  $\in$  $<G_{0},...,G_{k}>$. 
\end{lemma} 
 
Take $k=2$, $G_{0}=F_{0}(t)$, $G_{1}=F_{1}(t)$ and $G_{2}=F_{2}(t)$. Using Noether's Theorem with $Y=0$ and the fact that all polynomials involved are homogeneous, we have $\hat F_{1}(t)$ $\in$ $<F_{0}(t),F_{1}(t),F_{2}(t)>$. Since $deg(F_{0}(t))=deg(F_{1}(t))>deg(F_{2}(t))$, we conclude that $\hat F_{1}(t)=F_{1}(t)+g(t)F_{2}(t)$, where $g(t)$ is a homogeneous polynomial of degree $deg(F_{1}(t))-deg(F_{2}(t))$. Moreover observe that $V_{c}(t)=V(\hat F_{1}(t),F_{2}(t))=V(F_{1}(t),F_{2}(t))$, where $V(H_{1},H_{2})$ denotes the projective algebraic variety defined by $(H_{1}=H_{2}=0)$. Similarly for $V_{b}(t)$ we have that $\hat F_{2}(t)$ $\in$ $<F_{0}(t),F_{1}(t),F_{2}(t)>$.  On the other hand, since $\hat F_{2}(t)$ has the lowest degree, we can assume that $\hat F_{2}(t)=F_{2}(t)$.
   
In an analogous way we have that $\hat F_{0}(t)=F_{0}(t)+m(t)F_{1}(t)+n(t)F_{2}(t)$ for the polynomial $\hat F_{0}(t)$. Now observe that $V(\hat F_{0}(t),\hat F_{2}(t))=V(F_{0}(t)+m(t)F_{1}(t),F_{2}(t))$ where $m(t)\in \mathbb C$ satisfying $m(0)=0$. Hence we can define the family of polynomials as being $P_{0}(t)=F_{0}(t)+m(t)F_{1}(t)$, $P_{1}(t)=F_{1}(t)$ and $P_{2}(t)=F_{2}(t) $. This defines a  family of mappings $({f}_{t})_{t \in D_{\epsilon'}}:\mathbb P^{3} \DashedArrow[->,densely dashed    ]\mathbb P^{2}$, and $V_{a}(t)$, $V_{b}(t)$ and $V_{c}(t)$ are fibers of ${f}_{t}$ for fixed $t$. Observe that, for ${\epsilon'}$ sufficiently small, $({f}_{t})_{t \in D_{\epsilon'}}$ is generic in the sense of definition $3.2$, and its indeterminacy locus $I({f}_{t})$ is precisely $I(t).$ Moreover, since $Gen(3,\nu,1,\gamma)$ is open, we can suppose that this family $({f}_{t})_{t \in D_{\epsilon'}}$ is in $Gen\left(3,\nu,1,\gamma\right)$. This concludes the proof of proposition $5.10$. 
\end{proof}
\label{fbarra} We observe that this family can be considered also as a family of mappings $(\overline{f}_{t})_{t \in D_{\epsilon'}}:\mathbb P^{3} \DashedArrow[->,densely dashed    ]\mathbb P_{[\gamma,\gamma,1]}^2$, where $\overline{f}_{t}=(P_{0}(t),F_{1}(t),F_{2}(t))$ where $\mathbb P_{[\gamma,\gamma,1]}^2$ denotes the weighted projective plane with weights $(\gamma,\gamma,1)$. Moreover, using the map \begin{eqnarray*}
{f_w} : \mathbb P_{[\gamma,\gamma,1]}^2&\to& \mathbb P^2\\
         (x_{0}: x_{1}:x_{2})&\to& (x_{0}: x_{1}: x_{2}^{\gamma})
\end{eqnarray*}
we can factorize $f_{t}$ as being $f_{t}=f_{w}\circ \overline {f_{t}}$ as shown in the diagram below:
$$\small{\xymatrix{
  \mathbb{P}^3 \ar[rr]^{f_{t}} \ar[dr]_{\overline{f_{t}}} && \mathbb{P}^2 \\
  & {\mathbb  P_{[\gamma,\gamma,1]}}\ar[ur]^{f_{w}} }}$$

Now we will prove that the remaining curves $V_{\tau}(t)$ are also fibers of $f_{t}$. In the local coordinates $X(t)=(x_{0}(t),x_{1}(t),x_{2}(t))$ near some point of $I(t)$ we have that the vector field $S$ is diagonal and the components of the map ${f}_{t}$ are written as follows:
\begin{equation}\label{eq3}
P_{0}(t)=u_{0t}x_{0}(t)+x_{1}(t)x_{2}(t)h_{0t}
\end{equation}
\begin{equation*}
P_{1}(t)=u_{1t}x_{1}(t)+x_{0}(t)x_{2}(t)h_{1t}
\end{equation*}
\begin{equation*}
P_{2}(t)=u_{2t}x_{2}(t)+x_{0}(t)x_{1}(t)h_{2t}
\end{equation*}
where the functions $u_{it} \in \mathcal O^{*}(\mathbb C^3,0)$ and $h_{it} \in \mathcal O(\mathbb C^3,0) ,0\leq i \leq 2$. Note that when the parameter $t$ goes to $0$ the functions $h_{i}(t),0\leq i \leq 2$ also goes to $0$.
We want to show that an orbit of the vector field $S$ in the coordinate system $X(t)$ that extends globally like a singular curve of the foliation $\mathcal F_{t}$ is a fiber of ${f}_{t}$.

\begin{lemma}\label{casomenor} Any generic orbit of the vector field $S$ that extends globally as singular curve of the foliations $\mathcal F_{t}$ is also a fiber of $f_{t}$ for fixed $t$.
\end{lemma}

\begin{proof} To simplify the notation we will omit the index $t$. Let $\delta(s)$ be a generic orbit of the vector field $S$ (here by a generic orbit we mean an orbit that is not  a coordinate axis). We can parametrize $\delta(s)$ as $s\to(as^{\gamma},bs^{\gamma},cs)$, $a\neq0,b\neq0,c\neq0$. Without loss of generality we can suppose that $a=b=c=1$.
 We have $$f_{t}(\delta(s))=[(s^{\gamma}u_{0}+s^{(1+\gamma)}h_{0}):(s^{\gamma}u_{1}+s^{(1+\gamma)}h_{1}):(su_{2}+s^{2\gamma}h_{2})^\gamma].$$ Hence we can extract the factor $s^{\gamma}$ from $f_{t}(\delta(s))$ and we obtain \begin{equation}\label{eq4}f_{t}(\delta(s))=[(u_{0}+sh_{0}):(u_{1}+s^{l}h_{1}):(u_{2}+s^{2\gamma}h_{2})^\gamma].\end{equation}

Since $V_{\tau}$ is a fiber, $f_{0}(V_{\tau})=[d:e:f]\in \mathbb P^2$ with $d\neq0,e\neq0,f\neq0$. If we take a covering of $I(f)=\{p_{1},..., p_{\frac{\nu^3}{\gamma}}\}$ by small open balls $B_{j}(p_{j})$, $1\leq j\leq\frac{\nu^3}{\gamma}$, the set $V_{\tau}\backslash\cup_{j}B_{j}(p_{j})$ is compact. For a small deformation $f_{t}$ of $f_{0}$ we have that $f_{t}[V_{\tau}(t)\backslash\cup_{j}B_{j}(p_{j})(t)]$ stays near $f[V_{\tau}\backslash\cup_{j}B_{j}(p_{j})]$. Hence for $t$ sufficiently small the components of expression \ref{eq4} do not vanish both inside as well as outside of the neighborhood $\cup_{j}B_{j}(p_{j})(t)$. 
 
This implies that the components of $f_{t}$ do not vanish along each generic fiber that extends locally as a singular curve of the foliation $\mathcal F_{t}$. This is possible only if $f_{t}$ is constant along these curves. In fact, $f_{t}(V_{\tau}(t))$ is either a curve or a point. If it is a curve then it cuts all lines of $\mathbb P^2$ and therefore the components should be zero somewhere. Hence $f_{t}(V_{\tau}(t))$ is constant and we conclude that $V_{\tau}(t)$ is a fiber.  
 
Observe also that when we make a blow-up with weights $(\gamma,\gamma,1)$ at the points of $I(f_{t})$ we solve completely the indeterminacy points of the mappings $f_{t}$ for each $t$.
\end{proof}

\subsubsection{Part 2} Let us now define a family of foliations $(\mathcal {G}_{t})_{t\in D_{\epsilon}}, \mathcal {G}_{t} \in  \mathcal A$ (see Section \ref{section4}) such that  $\mathcal{F}_{t}= f_{t}^{ \ast}(\mathcal{G}_{t})$ for all $t\in D_{\epsilon}$. Firstly we consider the case $n=3.$
Instead of utilize the foliation $\mathcal F$ obtained as the foliation $f^*\mathcal G$, the idea that we will utilize in this part of the proof is to consider $\mathcal F$ on $\mathbb P^n$ defined as the foliation pull-back foliation from $\mathbb P^n$ to $\mathbb P_{[\gamma,\gamma,1]}^2$. 
 \begin{eqnarray*}
\overline{f}:\mathbb P^n&\to&\mathbb P_{[\gamma,\gamma,1]}\\
         \overline{f}^*\eta &\to& \eta.
\end{eqnarray*} once they define the same foliation.
 Let $M_{[\gamma,\gamma,1]}(t)$ be the family of ``complex algebraic threefolds" obtained from $\mathbb P^3$ by blowing-up with weights $(\gamma, \gamma,1)$ at the $\frac{\nu^{3}}{ \gamma}$ points $p_{1}(t),..., p_{j}(t),...,p_{\frac{\nu^{3}}{\gamma}}(t)$ corresponding to $I(t)$ of $\mathcal F_{t}$; and denote by $$\pi_{w}(t):M_{[\gamma,\gamma,1]}(t) \to \mathbb P^3$$ the blowing-up map. The exceptional divisor of $\pi_{w}(t)$ consists of ${\frac{\nu^{3}}{\gamma}}$ orbifolds \break $E_{j}(t)=\pi_{w}(t)^{-1}(p_{j}(t)),$ $1\leq j \leq {\frac{\nu^{3}}{\gamma}}$, which are weighted projective planes of the type $\mathbb P_{[\gamma,\gamma,1]}^2$.
 More precisely, if we blow-up $\mathcal{F}_{t}$ at the point $p_{j}(t)$, then the restriction of the strict transform $\pi_{w}^{\ast}\mathcal{F}_{t}$ to the exceptional divisor $E_{j}(t)=\mathbb P_{[\gamma,\gamma,1]}^2$ is the same quasi-homogeneous $1$-form that defines $\mathcal{F}_{t}$ at the point $p_{j}(t)$.
Using the map \begin{eqnarray*}
{f_w} : \mathbb P_{[\gamma,\gamma,1]}^2&\to& \mathbb P^2\\
         (x_{0}: x_{1}:x_{2})&\to& (x_{0}: x_{1}: x_{2}^{\gamma})
\end{eqnarray*}
it follows that we can push-forward the foliation to $\mathbb P^2$. 
Let us denote by \break$ \mathbb{F}ol'_{2}[d',2,(\gamma,\gamma,1)]$ the set of
 $\{\hat{\mathcal G}\}$  saturated foliations of degree $d'=\gamma(d+1)+1$ on $\mathbb P^2_{[\gamma,\gamma,1]}$ with one invariant line in general position and $Il_{1}(d,2)$  the subsets of saturated foliations with an invariant line in $\mathbb P^2$ respectively. The mapping ${f_w} : \mathbb P_{[\gamma,\gamma,1]}^2\to \mathbb P^2$ induces a natural isomorphism $({f_w})_*: Il_{1}(d,2) \to\mathbb{F}ol'_{2}[d',2,(\gamma,\gamma,1)]$. 
 With this process in mind we produce a family of holomorphic foliations in $\mathcal A\subset Il_{1}(d,2)$. This family is the ``holomorphic path'' of candidates to be a deformation of $\mathcal G_{0}$. In fact, since $(\mathcal {A}'={f_w})_*{(\mathcal A)}$  is an open set inside $ \mathbb{F}ol'_{2}[d',2,(\gamma,\gamma,1)]$ we can suppose that this family is inside $\mathcal A$. Hence using the mapping ${f_w}_*$ we can transport holomorphic from $\mathcal A$ to $\mathcal {A}'$ and vice-versa.

We fix the exceptional divisor $E_1(t)$ to work with and we denote by $\hat {\mathcal{G}}_t \in \mathcal {A}'$ the restriction of $\pi_w^*\mathcal{F}_t$ to $E_1(t)$. As we have seen, this process produces foliations in $\mathcal A'$ up to a linear automorphism of $\mathbb P_{[\gamma,\gamma,1]}^2$. Consider the family of mappings $\overline{f}_{t}:\mathbb P^{3} \DashedArrow[->,densely dashed    ]\mathbb P_{[\gamma,\gamma,1]}^2$, ${t \in D_{\epsilon'}}$ defined in Proposition \ref{recupmapas}. We will consider the family $(\overline{f}_{t})_{t\in D_{\epsilon}}$ as a family of rational maps  $\overline{f}_{t}:\mathbb P^3  \DashedArrow[->,densely dashed    ] E_{1}(t)$; we decrease $\epsilon$ if necessary. Note that the map $$\overline{f}_{t}\circ \pi_{w}(t):M_{[\gamma,\gamma,1]}(t) \backslash \cup_{j} E_{j}(t) \to E_{1}(t) \simeq P_{[\gamma,\gamma,1]}^2$$ extends holomorphically, that is, as an orbifold mapping, to 
$$\hat{f}_{t}:M_{[\gamma,\gamma,1]}(t) \to   E_{1}(t)\simeq \mathbb P_{[\gamma,\gamma,1]}^2.$$
This is due to the fact that each orbit of the vector field $S_t$ determines an equivalence class in $\mathbb P_{[\gamma,\gamma,1]}^2$ and is a fiber of the map $$(x_{0}(t),x_{1}(t),x_{2}(t))\to(x_{0}(t), x_{1}(t),x_{2}^\gamma(t)).$$

The mapping  $\overline{f}_t$ can be interpreted as follows. Each fiber of $\overline{f}_t$ meets $p_{j}(t)$ once, which implies that each fiber of $\hat{f}_{t}$ cuts $E_{1}(t)$ once outside of the singular line in $[M_{[\gamma,\gamma,1]}(t) \cap E_{1}(t)]$. Since $M_{[\gamma,\gamma,1]}(t) \backslash \cup_{j} E_{j}(t)$ is biholomorphic to $\mathbb P^3 \backslash I(t)$, after identifying $E_{1}(t)$ with $\mathbb P_{[\gamma, \gamma,1]}^2$, we can imagine that if $q \in M_{[\gamma,\gamma,1]}(t) \backslash \cup_{j} E_{j}(t)$ then $\hat{f}_{t}(q)$ is the intersection point of the fiber $\hat{f}_{t}^{-1}(\hat{f}_{t}(q))$ with $E_{1}(t)$. We obtain a mapping $$\hat{f}_{t}:M_{[\gamma,\gamma,1]}(t) \to \mathbb P_{[\gamma,\gamma,1]}.$$
It can be extended over the singular set of $M_{[\gamma,\gamma,1]}(t)$ using Riemann's  Extension Theorem. This is due to the fact that the orbifold $M_{[ \gamma,\gamma,1]}(t)$ has singular set of codimension $2$ and these singularities are of the quotient type; therefore it is a normal complex space. We shall also denote this extension by $\hat{f}_{t}$ to simplify the notation. We remark that the blowing-up with weights $(\gamma,\gamma,1)$ can completely solve the indeterminacy set of $\overline{f}_t$ or ${f}_t$ for each $t$ as the reader can check. 
With all these ingredients we can define the foliation $\tilde{\mathcal{F}}_{t}=\overline{f}_{t}^\ast(\hat {\mathcal{G}}_t
)={f}_{t}^\ast({\mathcal{G}}_t)\in PB(\Gamma-1,\nu,1,1,\gamma)$. This foliation is a deformation of $\mathcal{F}_0$.  Based on the previous discussion let us denote ${\mathcal{F}_{1}}({t})= \pi_{w}(t)^{\ast} ({\mathcal{F}}_{t})$ and $\hat{\mathcal{F}_{1}}({t})= \pi_{w}(t)^{\ast} (\tilde{\mathcal{F}}_{t})$.

\begin{lemma}If ${\mathcal{F}_{1}}({t})$ and $\hat{\mathcal{F}_{1}}({t})$ are the foliations defined previously, we have that $${\mathcal{F}_{1}}({t})|_{{E_{1}(t)}\simeq \mathbb P_{[\gamma,\gamma,1]}^2}={\hat {\mathcal G}_{t}}=\hat{\mathcal{F}_{1}}({t})|_{{E_{1}(t)}\simeq \mathbb P_{[\gamma,\gamma,1]}^2}$$ where ${\hat {\mathcal G}_{t}}$ is the foliation induced on ${E_{1}(t)}\simeq \mathbb P_{[\gamma,\gamma,1]}^2$ by the quasi-ho\-mo\-ge\-ne\-ous $1$-form $\eta_{p_{1}(t)}$. 
\end{lemma}

\begin{proof} In a neighborhood of $p_1(t) \in I(t)$, ${\mathcal{F}}_{t}$ is represented by the quasi-homogeneous $1$-form $\eta_{p_{1}(t)}$. This $1$-form satisfies $i_{S_t}\eta_{p_{1}(t)}=0$ and therefore naturally defines a foliation on the weighted projective space $E_{1}(t)\simeq \mathbb P_{[\gamma,\gamma,1]}^2$. This proves the first equality. The second equality follows from the geometrical interpretation of the mapping $\hat{f}_{t}:M_{[\gamma,\gamma,1]}(t) \to  \mathbb P_{[\gamma,\gamma,1]}^2 $, since $\hat{\mathcal{F}_{1}}({t})=\hat f_{t}^{*}(\hat{\mathcal {G}}_{t})$. \end{proof}
 
 Now we choose an afine chart of the space $\mathbb P_{[\gamma,\gamma,1]}^2$. This afine chart is biholomorphic to $\mathbb C^2$. In this affine chart for each $t$ the foliation $(\hat{\mathcal {G}}_{t})$ has $d^2$ singular points.
 
Let ${\tau}_{1}(t)$ be a singularity of $\hat{\mathcal G}_{t}$ outside of the line at infinity . Since the map $t \to {\tau}_{1}(t) \in \mathbb P_{[\gamma,\gamma,1}^2$ is holomorphic, there exists a holomorphic family of automorphisms of $\mathbb P_{[\gamma,\gamma,1}^2$, $t \to H(t)$ such that  ${\tau}_{1}(t)=[0:0:1]$ $\in {{E_{1}(t)}\simeq \mathbb P_{[\gamma,\gamma,1}^2}$  is kept fixed.  Observe that such a singularity has non algebraic separatrices at this point. Fix a local analytic coordinate system $(x_t,y_t)$ at 
${\tau}_{1}(t)$ such that the local separatrices are $(x_t=0)$ and $(y_t=0)$, respectively. Here we are considering the affine chart of $\mathbb P_{[\gamma,\gamma,1}^2$ which is biholomorphic to $\mathbb C^2$. This is useful because the foliations ${\mathcal G}_{t}$ and $\hat{\mathcal G}_{t}$ in this local coordinates are at least bihomolomorphic equivalents. Observe that the local smooth hypersurfaces along $\hat V_{\tau_1(t)}=\hat{f}_{t}^{-1}({\tau}_{1}(t))$ defined by $\hat X_t:=(x_t\circ\hat{f}_{t}=0)$ and 
$\hat Y_t:=(y_t\circ\hat{f}_{t}=0)$ are invariant for $\hat{\mathcal{F}_{1}}({t})$. Furthermore, they meet transversely along  $\hat V_{\tau_1(t)}$. On the other hand, $\hat V_{\tau_1(t)}$ is also contained in the Kupka set of ${\mathcal{F}_{1}}({t})$.  Therefore there are two local smooth hypersurfaces $X_t:=(x_t\circ\hat{f}_{t}=0)$ and 
$Y_t:=(y_t\circ\hat{f}_{t}=0)$ invariant for ${\mathcal{F}_{1}}({t})$ such that:
\begin{enumerate}
\item $X_t$ and $Y_t$ meet transversely along $\hat V_{\tau_1(t)}$.
\item$X_t\cap\pi_{w}(t)^{-1}(p_{1}(t))=(x_t=0)=\hat X_t\cap\pi_{w}(t)^{-1}(p_{1}(t))$ and $Y_t\cap\pi_{w}(t)^{-1}(p_{1}(t))=(y_t=0)=\hat Y_t\cap\pi_{w}(t)^{-1}(p_{1}(t))$ (because ${\mathcal{F}_{1}}({t})$ and $\hat{\mathcal{F}_{1}}({t})$) coincide on ${{E_{1}(t)}\simeq \mathbb P^2}$).
\item $X_t$ and $Y_t$ are deformations of $X_0=\hat X_0$ and $Y_0=\hat Y_0$, respectively. 
\end{enumerate}

\begin{lemma}\label{lemafund}$X_t=\hat X_t$ for small $t$.
\end{lemma}
\begin{proof} Let us consider the projection $\hat{f}_{t}:M_{[\gamma,\gamma,1]}(t) \to  \mathbb P_{[\gamma,\gamma,1]}^2$ on a neighborhood of the regular fibre  $\hat V_{\tau_1(t)}$, and fix local coordinates $x_t,y_t$ on $\mathbb P_{[\gamma,\gamma,1}^2$ such that $X_t:=(x_t\circ\hat{f}_{t}=0)$. For small $\epsilon$, let $H_{\epsilon}=(y_t\circ\hat{f}_{t}=\epsilon)$. Thus $\hat\Sigma_\epsilon=\hat X_t\cap H_\epsilon$ are (vertical) compact curves, deformations of $\hat\Sigma_0=\hat V_{\tau_1(t)}$. Set $\Sigma_\epsilon=X_t\cap\hat H_\epsilon$. The $\Sigma_\epsilon's$, as the $\hat\Sigma_\epsilon's$, are compact curves (for $t$ and $\epsilon$ small), since $X_t$ and $\hat X_t$ are both deformations of the same $X_0$. Thus for small $t$, $X_t$ is close to $\hat X_t$. It follows that $\hat{f}_{t}(\Sigma_\epsilon)$ is an analytic curve contained in a small neighborhood of ${\tau_1(t)}$, for small $\epsilon$. By the maximum principle, we must have that $\hat{f}_{t}(\Sigma_\epsilon)$ is a point, so that  $\hat{f}_{t}(X_t)=\hat{f}_{t}(\cup_{\epsilon}\Sigma_\epsilon)$ is a curve $C$, that is, $X_t=\hat{f}_{t}^{-1}(C)$. But $X_t$ and $\hat X_t$ intersect the exceptional divisor ${{E_{1}(t)}=\mathbb P_{[\gamma,\gamma,1}^2}$ along the separatrix $(x_t=0)$ of $\mathcal G_t$ through ${\tau_1(t)}$. This implies that $X_t=\hat{f}_{t}^{-1}(C)=\hat{f}_{t}^{-1}(x_t=0)=\hat X_t$.
\end{proof}
We have proved that the foliations $\mathcal F_t$ and $\tilde{\mathcal F_t}$ have a common local leaf: the leaf that contains $\pi_{w}(t)\left(X_t\backslash\hat V_{\tau_1(t)}\right)$ which is not algebraic. Let $D(t):=Tang(\mathcal{F}(t),\hat {\mathcal{F}}(t))$ be the set of tangencies between $\mathcal{F}(t)$ and $\hat {\mathcal{F}}(t)$. This set can be defined by $D(t)=\{ Z \in \mathbb {C}^4; \Omega(t) \wedge\hat {\Omega}(t)=0\},$ where  $\Omega(t)$ and  $\hat{\Omega}(t)$ define $\mathcal{F}(t)$ and $\hat{\mathcal{F}}(t)$, respectively. Hence it is an algebraic set. Since this set contains an immersed non-algebraic surface $X_t$, we necessarily have that  $D(t)=\mathbb P^3.$ This proves Theorem B in the case $n=3.$

Suppose now that $n\geq4$. The previous argument implies that if $\Upsilon$ is a generic $3-$plane in $\mathbb P^n$, we have $\mathcal{F}(t)_{|\Upsilon}=\hat{\mathcal{F}}(t)_{|\Upsilon}$. In fact, such planes cut transversely every strata of the singular set, and $I(t)$ consists of $\frac{\nu^{3}}{ \gamma} $ points. This implies that $f_{t}$ is generic for $|t|$ sufficiently small. We can then repeat the previous argument, finishing the proof of Theorem A.

Recall from Definition \ref{generic} the concept of a generic map. Let \break $f  \in BRM\left(n,\nu,1,1,\gamma\right)$, $I(f)$ its indeterminacy locus and $\mathcal F$ a foliation on $\mathbb P^n$, $n\geq3$. Consider the following properties:

\begin{center}
\begin{minipage}{10cm}
$\mathcal{P}_1:$ If n=3, at any point $p_{j}\in I(f)$ $\mathcal F$ has the following local structure: there exists an analytic coordinate system $(U^{p_{j}},Z^{p_{j}})$ around $p_{j}$ such that $Z^{p_{j}}(p_{j})=0 \in ({\mathbb C^3,0})$ and $\mathcal F|_{(U^{p_{j}},Z^{p_{j}})}$ can be represented by a quasi-homogenous $1$-form $\eta_{p_{j}}$ (as described in the Lemma \ref{lema5.6}) such that
\begin{enumerate} 
\item[(a)] $Sing(d \eta_{p_{j}}) = {0}$,
\item[(b)] $0$ is a quasi-homogeneous singularity of the type $\left [{\gamma}:{\gamma}:1\right ]$.
\end{enumerate}
If $n\geq4$, $\mathcal F$ has a local structure product: the situation for n=3 ``times'' a regular foliation in ${\mathbb C^{n-3}}$.
\end{minipage}
\end{center}

\begin{center}
\begin{minipage}{10cm}
$\mathcal{P}_2:$ There exists a fibre $f^{-1}(q)=V(q)$ such that $V(q)=f^{-1}(q)\backslash I(f)$ is contained in the Kupka-Set of $\mathcal F$ and $V(q)$ is not contained in $(F_{2}=0)$.
\end{minipage}
\end{center}

\begin{center}
\begin{minipage}{10cm}
$\mathcal{P}_3:$ $V(q)$ has transversal type $X$, where $X$ is a germ of vector field on $({\mathbb C^2,0})$ with a non algebraic separatrix and such that $ 0 \in{\mathbb C^2}$ is a non-degenerate singularity with eigenvalues $\lambda_{1}$ and $\lambda_{2}$, $\frac{\lambda_{2}}{\lambda_{1}} \notin {\mathbb R}$.
\end{minipage}
\end{center}

Lemma \ref{lemafund} allows us to prove the following results:

\begin{main}\label{teoc}
In the conditions above, if properties $\mathcal{P}_1$, $\mathcal{P}_2$ and $\mathcal{P}_3$ hold then $\mathcal F$ is a pull back foliation, $\mathcal {F}= f^{*}(\mathcal{G})$, where $\mathcal{G}$ is of degree $d\geq2$ on $\mathbb P^2$ with one invariant line.  
\end{main}

Let us denote by $ \mathbb{F}ol'_{2}[d',2,(\gamma,\gamma,1)]$ the set of
 $\{\hat{\mathcal G}\}$  saturated foliations of degree $d'$ on $\mathbb P^2_{[\gamma,\gamma,1]}$ with one invariant line.  According to this notation the previous Thereom can be re-written as:

\begin{main}\label{teoc}
In the conditions above, if properties $\mathcal{P}_1$, $\mathcal{P}_2$ and $\mathcal{P}_3$ hold then $\mathcal F$ is a pull back foliation, $\mathcal {F}= \overline{f}^{*}(\hat{\mathcal{G}})$, where $\hat{\mathcal{G}} \in \mathbb{F}ol'_{2}[d',2,(\gamma,\gamma,1)]$
\end{main}

\bibliographystyle{amsalpha}

\end{document}